\documentclass[reqno,12pt]{amsart}
\usepackage{amsmath, amsthm, amscd, amsfonts, amssymb, graphicx, color}
\usepackage[bookmarksnumbered, colorlinks, plainpages]{hyperref}
\title[ Approximation of fixed points for a representation]{ Approximation of fixed points for a representation of nonexpansive mappings in Banach spaces}
 \author{   Ebrahim  Soori  }
 \thanks{ \!\!\!\!\!\!\!\! \!\!2010 Mathematics Subject Classification: 47H09; 47H10. \\ E-mail address: sori.e@lu.ac.ir; sori.ebrahim@yahoo.com.}
\theoremstyle{plain}
\newtheorem{lem}{\textbf{Lemma}}[section]
\newtheorem{thm}[lem]{\textbf{Theorem}}

\newtheorem{de}[lem]{\textbf{Definition}}

\theoremstyle{definition}

\theoremstyle{definition}
\theoremstyle{remark}

\renewcommand{\sc}{\mathcal{S}}

\newcommand{\ud}{\,\mathrm{d}}
\begin{document}
\begin{large}

\maketitle
 \begin{center}
 \begin{normalsize}
    Department  of Mathematics, Lorestan University, Khoramabad, Lorestan, Iran.
 \end{normalsize}
 \end{center}
\begin{abstract}

\begin{normalsize}
The purpose of this paper is to study an implicit scheme    for  a representation of nonexpansive mappings
on a closed  convex subset of a smooth and uniformly convex Banach space with respect to a     left regular sequence of means
  defined on an appropriate space of bounded real valued functions of the semigroup.
 This algorithm   extends the algorithm that introduced  in [N. Hussain, M. L. Bami and E. Soori, An implicit method for finding a common
fixed point of a representation of
nonexpansive mappings in Banach spaces, Fixed Point Theory and Applications 2014, 2014:238].
\end{normalsize}
\end{abstract}
\begin{normalsize}
   \textbf{keywords}: Smooth Banach space;  Weakly sequentially continuous; Sunny nonexpansive retraction; Representation of nonexpansive mappings, Opial's condition.
   \end{normalsize}

\section{ Introduction}
Let $C$ be a nonempty closed and convex subset of a Banach space $E$ and
$E^{*}$ be the dual space of $E$. Let $\langle.,.\rangle$   denote the pairing between $E$ and $E^{*}$. The
normalized duality mapping $J: E \rightarrow E^{*}$
is defined by
\begin{align*}
    J(x)=\{f \in E^{*}: \langle x, f \rangle= \|x\|^{2}=\|f\|^{2} \}
\end{align*}
for all $x \in E.$ In the sequel, we use $j$ to denote the single-valued normalized
duality mapping. Let $U = \{x \in E : \|x\| = 1\}$. $E$ is said to be smooth or said
to   have a G$\hat{a}$teaux differentiable norm if the limit
\begin{align*}
    \lim_{t\rightarrow 0}\frac{\|x+ty\|-\|x\|}{t}
\end{align*}
exists for each $ x, y \in U$. $E$ is said to have a uniformly G$\hat{a}$teaux differentiable
norm if for each $y \in U$, the limit is attained uniformly for all $ x \in U$. $E$ is said
to be uniformly smooth or said to   have a uniformly F$r\acute{e}$chet differentiable
norm if the limit is attained uniformly for $x, y \in U$. It is known that if the
norm of $E$ is uniformly G$\hat{a}$teaux differentiable, then the duality mapping $J$
is single valued and uniformly norm to weak$^{*}$ continuous on each bounded
subset of $E$. A Banach space $E$ is   smooth if the duality mapping $J$ of
$E$ is single valued. We know that if $E$ is smooth, then $J$ is norm to weak-star continuous; for more details, see  \cite{tn}. We denote the strong convergence and
the weak convergence of a sequence $\{x_{n}\}$ to $x$ in $E$ by $x_{n} \rightarrow x$ and $x_{n} \rightharpoonup x$, respectively. We
also denote the weak$^{*}$ convergence of a sequence $\{x_{n}^{*}\}$  to $x^{*}$ in $E^{*}$ by  $x_{n}^{*}  \overset*  \rightharpoonup   x^{*} $. The duality
mapping $J$ is said to be weakly sequentially continuous if $x_{n} \rightharpoonup x$ implies that $Jx_{n} \overset* \rightharpoonup Jx$; see
 \cite{Ag} for details.

Let $C$ be a nonempty closed and convex subset of a Banach space $E$.  A mapping $ T$ of $ C $ into itself is called nonexpansive if $\|Tx - Ty\| \leq \|x - y\|,$ for all $x, y \in C$ and a mapping $f$ is an $\alpha$-contraction on $E $ if  $ \|f (x) -f (y)\| \leq \alpha \|x - y\|, \;x, y \in E$  such that  $0 \leq\alpha < 1$.

In this paper   we extend  theorem 3.1 in \cite{soo} by removing  the compactness condition of  $ C $      that is a very strong condition on it, where $ C $ is   a  closed  convex subset of a smooth and uniformly convex  Banach space. For this purpose, first we extend Lemma 3.5 in \cite{lmt} and  Lemma 1 in \cite{ss}, by removing the compactness condition on $C$   and we change   the   proofs of this  theorem and these Lemmas.

\section{preliminaries}
 Let $S$ be a semigroup. We denote by $B(S)$ the Banach space of all bounded real-valued functions defined on $ S $ with supremum
norm. For each $ s \in S$ and $f\in B(S)$ we define $l_{s}$ and  $r_{s}$ in $B(S)$  by\\ \;\indent$\quad (l_{s}f )(t) = f (st)$ ,\qquad $(r_{s}f )(t) = f (ts)$,\quad $ \;\;( t \in S)$.\\
Let $ X$ be a subspace of $B(S)$ containing 1 and let $ X^{*}$ be its topological dual. An element $\mu $ of $X^{*} $ is said to be a mean on $X$ if $\|\mu\|=\mu(1)=1$. We often write $\mu_{t}(f (t))$ instead of $ \mu(f )$ for $ \mu\in X^{*}$ and $f \in X$. Let $ X $ be left invariant (resp. right
invariant), i.e. $ l_{s}(X) \subset X$ (resp. $ r_{s}(X) \subset X$) for each $ s \in S$. A mean $\mu$ on $X$ is said to be left invariant (resp. right invariant) if $\mu(l_{s}f ) =\mu(f )$ (resp. $\mu(r_{s}f )= \mu(f )$) for each $s \in S $ and $f \in X$. $X $ is said to be left (resp. right) amenable if $X$ has a left (resp. right) invariant mean. $X$ is amenable if $X$ is both left and right amenable. As is well known, $B(S)$ is amenable when $S$ is a
commutative semigroup (see  page 29 of \cite{tn}). A net $\{\mu_{\alpha}\}$ of means on $X$ is said to be left regular if  $$\displaystyle\lim_{\alpha} \|l^{*}_{s}\mu_{\alpha}-\mu_{\alpha}\|= 0,$$  for
each $s \in S$, where $l_{s}^{*}$ is the adjoint operator of $l_{s}$.

Let $f$ be a  function of semigroup $S$ into a reflexive Banach space $E$ such that the weak closure of $\{f(t):t\in S\}$ is weakly
compact and let $X$ be a subspace of $B(S)$ containing all the functions $t \rightarrow \left<f(t),x^{\ast}\right>$ with  $ x^{\ast} \in E^{\ast}  $. We know from \cite{hi} that for any $ \mu\in X^{\ast}$, there
exists a unique element $ f_{\mu}$ in $E$ such that $\left<f_{\mu},x^{\ast}\right>=\mu_{t}\left<f(t),x^{\ast}\right>$
for all $ x^{\ast}\in E^{\ast} $. we denote such $f_{\mu}$ by $\int\!f(t)\ud \mu(t)$. Moreover, if $ \mu$ is a mean on $X$ then from \cite{ki}, $\int\! f(t)\ud \mu(t) \in \overline{\rm{co}}\,\{f(t):t\in S\}$.

Let $C$ be a nonempty closed and convex subset of $E$. Then, a family   $ \sc=\{T_{s}:s\in S\} $ of mappings from $C$ into itself is said to be a  representation of $S$ as nonexpansive mapping on $C$ into itself if $ \sc$ satisfies the following :\\
(1) $T_{st}x=T_{s}T_{t}x$ for all $s, t \in S$ and $x\in C$;\\
(2) for every $s \in S$ the  mapping $T_{s}: C \rightarrow C$ is nonexpansive.\\
 We denote by Fix($ \sc$)
the set of common fixed points of $\sc$, that is \\ Fix($ \sc$)=$\bigcap_{ s\in S}\{x\in C: T_{s}x=x  \}$.


\begin{thm}{\rm(\cite{said})}.\label{tu}
Let $ S $ be a semigroup, let $ C $ be a closed, convex subset of a
reflexive Banach space $  E$,  $ \sc=\{T_{s}:s\in S\} $ be a representation of $S$ as nonexpansive mapping from $C$ into
itself such that weak closure of  $\lbrace T_{t}x : t \in S \rbrace $ is weakly compact for each $ x \in C $ and $X$ be a subspace of $B(S)$ such that $ 1 \in X $ and the mapping $t \rightarrow \left<T (t)x, x^{*}\right>$ be an element of $ X $ for
each $x \in C$ and $ x^{*} \in E$, and $\mu$ be a mean on $X$.
If we write $T_{\mu}x $ instead of
$\int\!T_{t}x\, d\mu(t),\/$
 then the following hold.\/\\{\rm({i})}\,  $T_{\mu}$ is a nonexpansive mapping from $C$ into $C$.\/\\{\rm({ii})} $T_{\mu}x=x $  for each  $x \in Fix(\sc) $.\/\\{\rm({iii})} $T_{\mu}x \in \overline{co}\,\{T_{t}x:t\in S\}$ for each $x\in C$.\/ \\{\rm({iv})} If $  X$ is $ r_{s} $-invariant  for
  each $ s \in S $ and  $ \mu $ is right invariant, then $T_{\mu}T_{t} =T_{\mu} $ for each $ t \in S $.
\end{thm}
\textbf{Remark}: From, Theorem 4.1.6 in \cite{tn}, every uniformly convex Banach space is strictly convex and reflexive.

Let $D$ be a subset of $B$ where $B$ is a subset of a Banach space $E$ and let $P$ be a retraction of $B$ onto $D$, that is,
$Px = x$ for each $x \in D$. Then $P$ is said to be sunny, if for each $x \in B$ and $t\geq0$ with $Px + t (x - Px) \in B$,
$P(Px + t (x - Px)) = Px$.
A subset $D$ of $B$ is said to be a sunny nonexpansive retract of $B$ if there exists a sunny nonexpansive retraction $P$ of
$B$ onto $D$.
We know that if $E$ is smooth and $P$ is a retraction of $B$ onto $D$, then $P$ is sunny and nonexpansive if and
only if for each $x \in B$ and $z \in D$,
$\langle  x-Px\,,\,J(z-Px)\rangle    \leq 0$.\\
For more details, see  \cite{tn}.

\begin{de} \cite{Ag} A vector space $X$ is said to satisfy Opial's condition,  if for each sequence $\{x_{n}\}$
 in $X$ which converges weakly to
point $x \in X$,
\begin{equation*}
    \liminf_{n\rightarrow\infty}\|x_{n}-x\|<\liminf_{n\rightarrow\infty}\|x_{n}-y\|\quad(y\in X,\;y\neq x).\qquad\qquad\qquad\qquad\qquad
\end{equation*}
\end{de}
\begin{de} \cite{Ag}
  Let $C$ be a nonempty subset of a Banach space $X$ and $T :C \rightarrow X$ a mapping. Then $T$ is said to be demiclosed at $v \in X$ if for any sequence
$\{x_{n}\}$ in $C$ the following implication holds:\\
$x_{n} \rightharpoonup u \in C $ and $Tx_{n} \rightarrow v$ imply $Tu = v$.
\end{de}
 Throughout the rest of this paper, the open ball of radius $r $ centered at 0 is denoted by $ B_{r}$. Let   $C$ be a nonempty closed  convex subset of a     Banach space  $E$. For $\epsilon > 0 $ and a mapping $ T : C \rightarrow C$, we let $ F_{\epsilon}(T; D)$ be the set of $\epsilon$-approximate fixed points of $T$ for a subset $D$ of $C$, i.e. $ F_{\epsilon}(T; D) = \{x \in D : \|x - Tx\| \leq\epsilon\}$.

\section{Main result}

In this section, we deal with   a strong convergence approximation scheme for finding a common element of  the set of common fixed points of a representation of nonexpansive mappings that expresses  the introduced  algorithm  in theorem 3.1 in \cite{soo} by removing the very strong condition    compactness    on $C$ for   uniformly  convex Banach spaces.

First,    we express Lemma 3.5 in \cite{lmt} by removing the compactness condition on $C$ for    smooth and reflexive    Banach spaces  as follows.

\begin{thm}\label{lmt}
Let $S$ be a  semigroup. Suppose that $E$ is a real  smooth and   reflexive Banach space.  Let $C$ be a closed convex subset of $E$.  Let  $X$ be a left invariant subspace of $B(S)$ such that $1\in X$, and the function $t\mapsto \langle T_{t}x,x^{*}\rangle$ is an element of $X$ for each $x\in C$ and $x^{*}\in E^{*}$.  Suppose that $X$ is left amenable. Let   $ \sc=\{T_{s}:s\in S\}$ be a representation of $S$ as nonexpansive mappings from $C$ into itself such that $  \rm{Fix}(\mathcal{S})\neq \emptyset$. If $J$ is weakly sequentially continuous,  then $\rm{ Fix}(\sc) $ is a
sunny nonexpansive retract of $C$ and the sunny nonexpansive retraction of $C$ onto $\rm{ Fix}(\sc) $ is unique.
\end{thm}

\begin{proof}

From Theorem 3.2.8  in \cite{Ag}, $E$ has the Opial condition. As in the proof of Lemma 3.5 in \cite{lmt},  by the Banach contraction principle, we put
a sequence  $z_{n}$ in $C$ as follows:
\begin{align}\label{zn1l}
      z_{n}=\frac{1}{n}x+(1- \frac{1}{n})T_{\mu}z_{n}\quad ( n \in \mathbb{N}),
\end{align}
where,  $x \in C$ is fixed and   $\mu$ is a left invariant mean on $X$.
 As in the proof of Lemma 3.5 in \cite{lmt}, we have
 \begin{align}\label{gab}
  \lim_{n\rightarrow \infty}\|z_{n}-T_{\mu}z_{n}\|=0.
 \end{align}

Now, we show that $\{z_{n}\}$ is bounded:\\  let $p\in  \rm{Fix}(\mathcal{S}) $. Since from (ii) of theorem \ref{tu},  $T_{\mu}p = p$, we have
\begin{align*}
   \|z_{n}-p\|^{2}=&\langle\frac{1}{n}x+(1- \frac{1}{n})T_{\mu}z_{n}  -p\:,\:J(z_{n}-p)\rangle\\
                 =& \frac{1}{n} \langle x-p,J(z_{n}-p)\rangle
                  +\langle(1-\frac{1}{n})
                 \Big(T_{\mu}z_{n}-T_{\mu}p\Big)\:,\:J(z_{n}-p)\rangle\\
                 \leq&(1-\frac{1}{n})\|z_{n}-p\|^{2} +\frac{1}{n} \langle x-p,J(z_{n}-p)\rangle.
\end{align*}
Thus,
\begin{equation}\label{6l}
\|z_{n}-p\|^{2}\leq \left< x-p\:,\:J(z_{n}-p)\right>.
\end{equation}
Hence,
\begin{equation*}
   \|z_{n}-p\| \leq\| x-p\|.
\end{equation*}
That is, the sequence $ \{z_{n}\}$ is bounded.

Next, We shall show that the sequence $\{z_{n}\}$ weakly converges   to an element of $\rm{Fix}(\mathcal{S}) $. In the other words,  we   show that the weak  limit set of $\{z_{n}\}$ which is denoted by $\omega_{\omega}\{z_{n}\}$ is a subset of $\rm{Fix}(\mathcal{S})$.
  Let $x^{*} \in \omega_{\omega}\{z_{n}\} $ and let $\{z_{n_{j}}\}$ be a subsequence of $ \{z_{n}\} $ such that $z_{n_{j}} \rightharpoonup x^{*}$. We need to show that $ x^{*} \in \rm{Fix}(\mathcal{S}) $. Since from  Corollary 5.2.10 in \cite{Ag},  $I -T_{t}$ is demiclosed at zero, for each $t \in S$ and hence  from   \eqref{gab}, we conclude that $x^{*} \in \rm{Fix}(\mathcal{S})$. Therefore,  $\omega_{\omega}\{z_{n}\} \subseteq \rm{Fix}(\mathcal{S})$.
Note that,        since   $ \{z_{n}\} $ is bounded  and    $E$ is reflexive,   from theorem 1.9.21 in \cite{Ag}, $ \{z_{n}\} $  is a weakly compact subset of $E$, hence   by  Proposition 1.7.2 in \cite{Ag},   we can select a subsequence $ \{z_{n_{j}}\}$
 of $ \{z_{n}\}$  such that
 $\{z_{n_{j}}\}$ weakly converges  to a point $z$. Therefore, from the above discussion,  $z \in \rm{Fix}(\mathcal{S})$.
Let $\{z_{n_{i}}\}$ and $\{z_{n_{j}}\}$ be subsequences of $ \{z_{n}\} $ such that $\{z_{n_{i}}\}$ and $\{z_{n_{j}}\}$ converge
weakly to $y$ and $z$, respectively.  Therefore,     $y,z \in \rm{Fix}(\mathcal{S})$. Since $J$ is weakly sequentially continuous, from \eqref{6l}, we have, $\{z_{n_{i}}\}$ and $\{z_{n_{j}}\}$
  converge strongly to $y$ and $z$, respectively. Indeed, from \eqref{6l} we have
  \begin{equation*}
\lim_{i\rightarrow \infty}\|z_{n_{i}}-y\|^{2}\leq \lim_{i\rightarrow \infty} \left< x-y\:,\:J(z_{n_{i}}-y)\right>=\left< x-y\:,\:J(y-y)\right>=0,
\end{equation*}
hence, $z_{n_{i}}\rightarrow y$ and similarly  $z_{n_{j}} \rightarrow z$.
   As in the proof of Lemma 3.5 in \cite{lmt}, we have, for each
$z \in \rm{Fix}(\mathcal{S})$ and $n \in N$,
 \begin{align}\label{la2}
 \langle  z_{n}-x\,,\,J(z_{n}-z)\rangle    \leq 0.
\end{align}

 Further,
 \begin{align*}
  | \langle  z_{n_{i}}-x\,&,\,J(z_{n_{i}}-z)\rangle  -\langle  y-x\,,\,J(y-z)\rangle | \nonumber \\=&  | \langle  z_{n_{i}}-x\,,\,J(z_{n_{i}}-z)\rangle -\langle  y-x\,,\,J(z_{n_{i}}-z)\rangle \nonumber \\&+ \langle  y-x\,,\,J(z_{n_{i}}-z)\rangle - \langle y-x\,,\,J(y-z)\rangle | \nonumber \\ \leq & | \langle  z_{n_{i}}-y\,,\,J(z_{n_{i}}-z)\rangle| + |\langle  y-x\,,\,J(z_{n_{i}}-z) -J(y-z)\rangle | \nonumber \\ \leq & \parallel  z_{n_{i}}-y \parallel \parallel J(z_{n_{i}}-z)\parallel + |\langle  y-x\,,\,J(z_{n_{i}}-z) -J(y-z)\rangle | \nonumber \\ \leq & \parallel  z_{n_{i}}-y \parallel M + |\langle  y-x\,,\,J(z_{n_{i}}-z) -J(y-z)\rangle |,
 \end{align*}
 where $M$ is an upper bound for $\{J(z_{n_{i}}-z)\}_{i \in \mathbb{N}}$. Hence, we have
  \begin{equation}\label{jish}
    \lim_{i\rightarrow \infty}  \langle  z_{n_{i}}-x\,,\,J(z_{n_{i}}-z)\rangle  = \langle  y-x\,,\,J(y-z)\rangle,
  \end{equation}
  Now, since $J$ is weakly sequentially continuous, from  \eqref{jish} and  \eqref{la2}, we have
\begin{align*}
    \langle  y-x\,,\,J(y-z)\rangle=\lim_{i\rightarrow \infty}\langle  z_{n_{i}}-x\,,\,J(z_{n_{i}}-z)\rangle    \leq 0.
\end{align*}

Similarly, we have $ \langle  z-x\,,\,J(z-y)\rangle  \leq 0$  and hence $y = z$. Indeed, we have
 \begin{align*}
   \parallel y-z \parallel ^{2}&=  \langle  y-z\,,\,J(y-z)\rangle = \langle  y-x\,,\,J(y-z)\rangle +  \langle  x-z\,,\,J(y-z)\rangle \\ =& \langle  y-x\,,\,J(y-z)\rangle +  \langle  z-x\,,\,J(z-y)\rangle  \leq 0,
 \end{align*}
 and hence $y = z$.  Thus, $ \{z_{n}\} $ weakly converges   to an element of  $ \rm{Fix}(\mathcal{S})$.

Let
us define a mapping $P$ of $C$ into itself by $Px = weak-\displaystyle\lim_{n}z_{n}$ i.e,    $   z_{n} \rightharpoonup Px$. Hence, $z_{n} \rightarrow Px $. Then, from the condition that the duality mapping is weakly sequentially continuous, as the above discussion, we have, for each $ z \in  \rm{Fix}(\mathcal{S}) $,
\begin{align*}
\langle  x-Px\,,\,J(z-Px)\rangle = \lim_{n\rightarrow\infty}\langle z_{n}- x\,,\,J(z_{n}-z)\rangle   \leq 0.
\end{align*}
It follows from   Lemma 5.1.6 in \cite{tn}, that $P$ is a sunny nonexpansive retraction of $C $ onto $\rm{Fix}(\mathcal{S}) $.  As in the Proof of Lemma 3.5. in \cite{lmt}, we have $P$ is unique.

\end{proof}

Now, we prove another version of  Lemma 1 in \cite{ss}, by removing the compactness condition on $C$ for  uniformly  convex  Banach spaces  as follows.

\begin{thm}\label{sunyn}
 Let $S$ be a  semigroup. Suppose that $E$ is a real   uniformly  convex    and smooth Banach space.  Let $C$ be a closed convex subset of   $E$.  Suppose that   $ \sc=\{T_{s}:s\in S\}$ be a representation of $S$ as nonexpansive mappings from $C$ into itself  such that $  \rm{Fix}(\mathcal{S})\neq \emptyset$. Let $X$ be a left invariant subspace of $B(S)$ such that $1\in X$, and the function $t\mapsto \langle T_{t}x,x^{*}\rangle$ is an element of $X$ for each $x\in C$ and $x^{*}\in E^{*}$.  If $ \mu $  is a left invariant mean on  $ X $. If $J$ is weakly sequentially continuous, then  $\rm{Fix}(T_{\mu})=T_{\mu}(C)= \rm{Fix}(\mathcal{S})$ and
  there
exists a unique sunny nonexpansive retraction from $ C $ onto $\rm{ Fix}(\sc) $.
\end{thm}
\begin{proof}
From Theorem 3.2.8  in \cite{Ag}, $E$ has the Opial condition. Consider  $t \in S$  and let   $x \in C$.
Let $p$  be an arbitrary element of $\rm{Fix}(\mathcal{S})$. Set
$D=\{y \in C: \|y-p\|\leq \parallel x-p \parallel\}$. Note that, D is a bounded closed convex set and  $x \in D$ and $T_{t}(D)\subset D$. Let
$\epsilon>0$. By   Theorem 1.2 in \cite{rb} there exists $\delta>0$ such that $\overline{\text{co}}F_{\delta}(T_{t} ; D)   \subset
F_{\epsilon}(T_{t} ; D)$. By   Corollary 1.1 in \cite{rb}, there also exists a natural number $N$ such
that
\begin{align*}
    \Big\|\frac{1}{N+1}\sum_{i=0}^{N}T_{t^{i}s}y-T_{t}\Big(\frac{1}{N+1}\sum_{i=0}^{N}T_{t^{i}s}y \Big)\Big\| \leq \delta,
\end{align*}
for all $  s \in S $ and $y \in D$. So we have
\begin{align*}
    \Big\|\frac{1}{N+1}\sum_{i=0}^{N}T_{t^{i}s}x-T_{t}\Big(\frac{1}{N+1}\sum_{i=0}^{N}T_{t^{i}s}x \Big)\Big\| \leq \delta,
\end{align*}
for all $  s \in S $.
   As in the proof of  Lemma 1 in \cite{ss},  Since  $\mu$ is left invariant,     we have
\begin{align*}
T_{\mu}x=& \int\! T_{s}x\ud \mu(s)=  \frac{1}{N+1}\sum_{i=0}^{N}\int\! T_{t^{i}s}x\ud \mu(s) \\=&\int\! \frac{1}{N+1}\sum_{i=0}^{N}T_{t^{i}s}x\ud \mu(s),
\\ \in \;& \overline{\text{co}}\left\{\frac{1}{N+1}\sum_{i=0}^{N}T_{t^{i}}(T_{s}x):s\in S\right\} \\ \subset & \; \overline{\text{co}}F_{\delta}(T_{t} ; D)  \subset
F_{\epsilon}(T_{t} ; D).
\end{align*}
Since  $\epsilon > 0$ was arbitrary,   we have
$
T_{t} T_{\mu}x=T_{\mu}x$, and hence $\rm{Fix}(T_{\mu})=T_{\mu}(C)= \rm{Fix}(\mathcal{S})$. The
rest of the proof of the other assertions is the same as the proof of theorem \ref{lmt}.

\end{proof}

Next, we prove another version of  Theorem 3.1 in   \cite{soo}, by removing the compactness condition on $C$ for  uniformly  convex  Banach spaces  as follows.

\begin{thm}\label{g1}
 Let $S$ be a  semigroup. Suppose that $E$  is a real    uniformly  convex     and smooth Banach space  and    $C$ be a nonempty  closed convex subset of $E$.  Suppose that   $ \sc=\{T_{s}:s\in S\}$ is a representation of $S$ as nonexpansive mappings from $C$ into itself  such that weak closure of  $\lbrace T_{t}x : t \in S \rbrace $ is weakly compact for each $ x \in C $
 and  $  \rm{Fix}(\mathcal{S})\neq \emptyset$.
 Let  $X$ be a left invariant subspace of $B(S)$ such that $1\in X$, and the function $t\mapsto \langle T_{t}x,x^{*}\rangle$ is an element of $X$ for each $x\in C$ and $x^{*}\in E^{*}$. Let $\{\mu_{n}\}$ be a left regular sequence of  means on $X$.  Suppose that $f$ is an $\alpha$-contraction on $ C$. Let $\epsilon_{n}$ be a sequence in $(0, 1)$ such that $\displaystyle \lim_{n} \epsilon_{n}=0$.  Let the duality mapping $J$ be  weakly sequentially continuous.
 Then    there exists a unique sunny nonexpansive retraction $ P $  of $ C $ onto $  \rm{Fix}(\mathcal{S})$ and $ x \in C $ such that the following  sequence  $\{z_{n}\}$   generated by
\begin{equation}\label{4}
    z_{n}=\epsilon_{n} f(z_{n})+(1-\epsilon_{n})T_{\mu_{n}}z_{n}\quad ( n \in \mathbb{N}),
\end{equation}
  strongly converges to $Px$.
\end{thm}
\begin{proof}

From Theorem 3.2.8  in \cite{Ag}, $E$ has the Opial condition.
Note that, from  theorem 2.2.8 in \cite{Ag}, $E$  is reflexive.

We  divide the proof into six   steps.

  Step 1.  The existence of $z_{n}$ which satisfies \eqref{4}.

  Proof.  The proof is the same as     the proof of Step 1 of theorem 3.1 in  \cite{soo}.
Step 2. $\{z_{n}\}$ is bounded.\\Proof. Let $p\in  \rm{Fix}(\mathcal{S}) $. Since from (ii) of theorem \ref{tu},  $T_{\mu_{n}}p = p$ for each $n \in \mathbb{N}$, we have
\begin{align*}
   \|z_{n}-p\|^{2}=&\langle\epsilon_{n}  f(z_{n})+(1-\epsilon_{n})T_{\mu_{n}}z_{n}-p\:,\:J(z_{n}-p)\rangle\\
                 =&\epsilon_{n} \langle f(z_{n})-f(p),J(z_{n}-p)\rangle+\epsilon_{n}\langle  f(p)-p\:,\:J(z_{n}-p)\rangle \\
                 &+\langle(1-\epsilon_{n})
                 \Big(T_{\mu_{n}}z_{n}-T_{\mu_{n}}p\Big)\:,\:J(z_{n}-p)\rangle\\
                 \leq&\epsilon_{n}  \alpha\|z_{n}-p\|^{2}+(1-\epsilon_{n})\|z_{n}-p\|^{2}\\&+\epsilon_{n}\Big \langle  f(p)-p\:,\:J(z_{n}-p)\Big\rangle.
\end{align*}
Thus,
\begin{equation}\label{6}
\|z_{n}-p\|^{2}\leq\frac{1}{1-\alpha}\left< f(p)-p\:,\:J(z_{n}-p)\right>.
\end{equation}
Hence,
\begin{equation*}
   \|z_{n}-p\|\leq\frac{1}{1-\alpha}\| f(p)-p\|.
\end{equation*}
That is, the sequence $ \{z_{n}\}$ is bounded.

  Step 3.  $\lim_{n\rightarrow \infty}\|z_{n}-T_{t}z_{n}\|=0 $,  for all $t \in S.$\\
Proof.  Consider  $t \in S$.
Let $p$  be an arbitrary element of $\rm{Fix}(\mathcal{S})$. Set
$D=\{y \in C: \|y-p\|\leq \frac{1}{1-\alpha}\| f(p)-p\|\}$. We remark that D is a bounded closed convex set and  $\{z_{n}\}\subset D$ and $T_{t}(D)\subset D$. Let
$\epsilon>0$. By   Theorem 1.2 in \cite{rb} there exists $\delta>0$ such that
 \begin{equation}\label{nonem}
   \overline{\text{co}}F_{\delta}(T_{t} ; D)  +B_{\delta}  \subset
F_{\epsilon}(T_{t} ; D).
 \end{equation}
  By   Corollary 1.1 in \cite{rb},  there also exists a natural number  $N$ such
that
\begin{equation}\label{12}
    \Big\|\frac{1}{N+1}\sum_{i=0}^{N}T_{t^{i}s}y-T_{t}\Big(\frac{1}{N+1}\sum_{i=0}^{N}T_{t^{i}s}y \Big)\Big\| \leq \delta,
\end{equation}
for all $  s \in S $ and $y \in D$. Let   $M_{0}=\frac{1}{1-\alpha}\| f(p)-p\|+\|p\|$. Therefore, $\displaystyle\sup_{y \in D}\|y\|\leq M_{0}$. Since $\{\mu_{n}\} $ is strongly left regular, there exists $N_{0} \in \mathbb{N} $ such that
$\|\mu_{n}-l^{*}_{t^{i}}\mu_{n}\|\leq \frac{\delta}{M_{0}}$ for $n \geq N_{0} $ and $i = 1, 2, \cdots , N. $ Then, we have
\begin{align}
\sup_{y \in D} \Big \|&T_{\mu_{n}}y-\int\! \frac{1}{N+1}\sum_{i=0}^{N}T_{t^{i}s}y\ud \mu_{n}(s) \Big  \|  \nonumber \\
=&\sup_{y\in D}\sup_{\|x^{*}\|=1} \Big | \langle T_{\mu_{n}}y ,x^{*}\rangle-\Big \langle \int\!\frac{1}{N+1}\sum_{i=0}^{N}T_{t^{i}s}y\ud \mu_{n}(s),x^{*}  \Big \rangle \Big  | \nonumber\\
   = &\sup_{y\in D}\sup_{\|x^{*}\|=1} \Big |\frac{1}{N+1}\sum_{i=0}^{N}(\mu_{n})_{s}\langle T_{s}y,x^{*}\rangle-\frac{1}{N+1}\sum_{i=0}^{N}(\mu_{n})_{s}\langle  T_{t^{i}s}y,x^{*}\rangle \Big  | \nonumber \\
   \leq & \frac{1}{N+1}\sum_{i=0}^{N}\sup_{y\in D}\sup_{\|x^{*}\|=1} \Big|(\mu_{n})_{s}\langle T_{s}y,x^{*}\rangle-(l^{*}_{t^{i}}\mu_{n})_{s}\langle T_{s}y,x^{*}\rangle \Big|\nonumber\\
   \leq &\max_{i=1,2,\cdots,N}\|\mu_{n}-l^{*}_{t^{i}}\mu_{n}\|(M_{0})\nonumber\\
   \leq &\max_{i=1,2,\cdots,N}\|\mu_{n}-l^{*}_{t^{i}}\mu_{n}\|(M_{0})\nonumber \\ \leq&\delta\quad \rm{( n\geq N_{0})}.\label{13}
\end{align}
By Theorem \ref{tu} we have
\begin{equation}\label{14}
\int\!\frac{1}{N+1}\sum_{i=0}^{N}T_{t^{i}s}y\ud \mu_{n}(s)\in \overline{\text{co}}\left\{\frac{1}{N+1}\sum_{i=0}^{N}T_{t^{i}}(T_{s}y):s\in S\right\}.
\end{equation}
It follows from \eqref{12}-\eqref{14} that
\begin{align*}
 T_{\mu_{n}}y =& \int\!\frac{1}{N+1}\sum_{i=0}^{N}T_{t^{i}s}y\ud \mu_{n}(s)+\Big(  T_{\mu_{n}}y -\int\!\frac{1}{N+1}\sum_{i=0}^{N}T_{t^{i}s}y\ud \mu_{n}(s) \Big) \\  \in & \overline{\text{co}}\left\{\frac{1}{N+1}\sum_{i=0}^{N}T_{t^{i}s}y:s\in S\right\}+B_{\delta} \\
\subset & \overline{\text{co}}F_{\delta}(T_{t} ; D)  +B_{\delta}  \subset
F_{\epsilon}(T_{t} ; D),
\end{align*}
for all $y \in D$ and $n \geq N_{0}$. Therefore, \\
$
\displaystyle \limsup_{n\rightarrow\infty}\sup_{y\in D}\|T_{t}(T_{\mu_{n}}y)-T_{\mu_{n}}y\|\leq \epsilon.
$
Since $\epsilon > 0 $ is arbitrary, we have
\begin{align}\label{tun}
 \limsup_{n\rightarrow\infty}\sup_{y\in D}\|T_{t}(T_{\mu_{n}}y)-T_{\mu_{n}}y\|=0.
\end{align}
Let $t \in S $ and $\epsilon > 0$, then there exists $\delta > 0$, which satisfies \eqref{nonem}. Take $L_{0}=\frac{2}{1-\alpha}\| f(p)-p\|$. Now from the condition $\displaystyle \lim_{n} \epsilon_{n}=0$ and from \eqref{tun}   there exists a natural number $N_{1}$ such that $T_{\mu_{n}}y \in F_{\delta}(T_{t} ; D)$ for all $ y \in D$ and $\epsilon_{n}<\frac{\delta}{2L_{0}}$ for all $n \geq N_{1}$. Since  $ p \in \rm{Fix(\sc) }$  and  $\{z_{n}\}\subset D$, we have
\begin{align*}
\epsilon_{n}\|  f(z_{n})-& T_{\mu_{n}}z_{n}\|\\\leq & \epsilon_{n}\Big(\|  f(z_{n})-  f(p)\|+\|  f(p)- p\|\\& +\| T_{\mu_{n}}p- T_{\mu_{n}}z_{n}\|\Big)\\\leq &
 \epsilon_{n}\Big(  \alpha\|z_{n}-p\|+\|  f(p)- p\|+\|z_{n}-p\|\Big)
 \\\leq &  \epsilon_{n}\left(\frac{1+  \alpha}{1-\alpha}\| f(p)-p\|+\|  f(p)- p\|\right)
\\=&\epsilon_{n}L_{0}\leq\frac{\delta}{2},
\end{align*}
for all $ n \geq N_{1}$. Observe that
\begin{align*}
z_{n}&=\epsilon_{n}  f(z_{n})+(1-\epsilon_{n})T_{\mu_{n}}z_{n}\\
&=T_{\mu_{n}}z_{n}+\epsilon_{n}\left(  f(z_{n})- T_{\mu_{n}}z_{n}\right)\\
&\in F_{\delta}(T_{t} ; D)+B_ {\frac{\delta}{2}} \\
&\subseteq F_{\delta}(T_{t} ; D)+B_{\delta}\\
&\subseteq F_{\epsilon}(T_{t} ; D).
\end{align*}
for all $ n \geq N_{1}$. This show that\\
\indent \quad$\|z_{n}-T_{t}z_{n}\|\leq\epsilon\quad \rm{ (n \geq N_{1})}$.\\
Since $\epsilon > 0 $ is arbitrary, we get
$\lim_{n\rightarrow \infty}\|z_{n}-T_{t}z_{n}\|=0 $.

 Step 4.    The weak  limit set of $\{z_{n}\}$ which is denoted by $\omega_{\omega}\{z_{n}\}$ is a subset of $\rm{Fix}(\mathcal{S})$.

Proof.  Let $x^{*} \in \omega_{\omega}\{z_{n}\} $ and let $\{z_{n_{j}}\}$ be a subsequence of $ \{z_{n}\} $ such that $z_{n_{j}} \rightharpoonup x^{*}$. We need to show that $ x^{*} \in \rm{Fix}(\mathcal{S}) $. Since from  Corollary 5.2.10 in \cite{Ag},  $I -T_{t}$ is demiclosed at zero, for each $t \in S$ and hence  from  step 3, we conclude that $x^{*} \in \rm{Fix}(\mathcal{S})$.

  Step 5.   There exists a unique sunny nonexpansive retraction $ P $  of $ C $ onto $  \rm{Fix}(\mathcal{S})$ and $ x \in C $ such that
\begin{align}\label{gbs}
  \Gamma  :=\limsup_{n}\langle x-Px\,,\,J(z_{n}-Px)\rangle\leq0.
\end{align}
Proof. Let $\mu$ be a cluster point  of $\{\mu_{n}\}$. It is obvious that $\mu$ is an invariant mean. Hence, by Lemma \ref{sunyn}   there exists a unique sunny nonexpansive retraction $ P $ of $ C $ onto
$ \rm{Fix(\sc)} $.  Banach Contraction Mapping Principle guarantees that   $fP$ has a unique fixed point $x\in C$.  We show that
\begin{equation*}
  \Gamma  :=\limsup_{n}\langle x-Px\,,\,J(z_{n}-Px)\rangle\leq0.
\end{equation*}

Note that, from the definition of
 $ \Gamma $  and      from Step 2 that $ \{z_{n}\} $ is bounded, and  since $E$ is reflexive, therefore from theorem 1.9.21 in \cite{Ag}, $ \{z_{n}\} $  is weakly compact, hence   by  Proposition 1.7.2 in \cite{Ag},   we can select a subsequence $ \{z_{n_{j}}\}$
 of $ \{z_{n}\}$ with the following properties:\\
(i)  $\displaystyle\lim_{j}\langle x-Px\,,\,J(z_{n_{j}}-Px)\rangle=\Gamma$;\\
(ii) $\{z_{n_{j}}\}$ weakly converges  to a point $z$;\\
by Step 4, we have $ z \in  \rm{Fix}(\mathcal{S}) $.  Since $E  $  is smooth,   and the duality mapping is weakly sequentially continuous, we have
\begin{align*}
\Gamma =& \lim_{j}\langle  x-Px\,,\,J(z_{n_{j}}-Px)\rangle=\langle  x-Px\,,\,J(z-Px)\rangle \leq 0.
\end{align*}

Step 6. $ \{z_{n}\}$  strongly converges to $Px$.\\
Proof.
 As in the proof of Theorem 3.1 in \cite{soo}, we have, for each $n\in \mathbb{N}$,
\begin{align}\label{gbh}
 \Vert z_{n}-Px \Vert^{2}\leq \frac{2}{1-   \alpha }\langle x-Px\, , \,J(z_{n}-Px)\rangle.
\end{align}
 Therefore, from \eqref{gbs}, \eqref{gbh} and that $ Px \in  \rm{Fix}(\mathcal{S})$, we conclude
\begin{align*}
    \limsup_{n}\|z_{n}-Px\|^{2}\leq &  \frac{2}{1-   \alpha}\limsup_{n}\langle x-Px \,,\,J(z_{n}-Px)\rangle\leq0.
\end{align*}
That is $z_{n} \rightarrow Px$.
\end{proof}

\begin{center}
\end{center}

\end{large}
\end{document}